\documentclass{amsproc}
\usepackage{amsfonts}
\usepackage{amsmath,amssymb}
\usepackage{graphicx}
\usepackage{setspace}
\usepackage[left=2cm,top=1.25cm,right=2cm]{geometry}

\newtheorem{lemma}{Lemma}

\newtheorem{theorem}{Theorem}
\numberwithin{equation}{section}
\begin{document}

\title{Weighted Approximation theorem for Choldowsky generalization of the q-Favard-Sz\'{a}sz operators}
\maketitle
\begin{center}
{\bf Preeti Sharma$^{1,\star}$, Vishnu Narayan Mishra$^{1,2,\dag}$  }\\ 
$^{1}$Department of Applied Mathematics and Humanities,
Sardar Vallabhbhai National Institute of Technology,
Ichchhanath Mahadev Dumas Road, Surat, Surat-395 007 (Gujarat), India.\\
$^{2}$L. 1627 Awadh Puri Colony Beniganj,
Phase -III, Opposite - Industrial Training Institute (I.T.I.),\\ Ayodhya Main Road,
Faizabad-224 001, (Uttar Pradesh), India.
\end{center}
\begin{center}
$^\star$Email: preeti.iitan@gmail.com\\
$^\dag$Email: vishnunarayanmishra@gmail.com,vishnu\_narayanmishra@yahoo.co.in, v\_n\_mishra\_hifi@yahoo.co.in, vishnu\_narayanmishra@rediffmail.com
\end{center}
\footnotetext[2]{Corresponding author}

\begin{abstract}
we study the convergence of these operators in a weighted space of functions on a positive semi-axis and estimate the approximation by using a new type of weighted modulus of continuity and error estimation.

Keywords: $q$-Favard- Sz\'{a}sz operators, error estimation.\\

\end{abstract}
\maketitle

\section{Introduction and auxiliary results}
The classical Favard--Sz\'{a}sz--perators are given as follows
\begin{equation}
S_n(f,;)= e^{-n x}\sum_{k=0}^{\infty} \frac{(n x)^k}{k!} f(\frac{k}{n})
\end{equation}
These operators and the generalizations have been studied by several other researcher  ( see. \cite{gal2008}-\cite{otto1950}) and references there in. In 1969, Jakimovski and Leviatan \cite{Jakimovski} introduced the Favard-Sz\'{a}sz type operator, by using Appell polynomials $p_k (x)(k \geq 0)$ defined by
$$ g(u)e^{-ux}=\sum_{k=0}^{\infty}p_k(x) u^{k},$$
where $g(z)=\sum_{n=0}^{\infty} a_n z^{n}$ is analytic function in the disc $|z|<R, R>1$ and $g(1)\neq 0,$
$$P_{n,t}(f,x)=\frac{e^{nx}}{g(1)}\sum_{k=0}^{\infty}p_k(nx)f\big(\frac{k}{n}\big)$$
and they investigated some approximation properties of these operators.\\
Atakut at el.\cite{Atakut2014} defined a choldowsky type of Favard--Sz\'{a}sz operators as follows:

\begin{equation}\label{nit1}
P_{n}^{*}(f,x)=\frac{e^{\frac{nx}{b_n}}}{g(1)} \sum_{k=0}^{\infty}p_k(\frac{nx}{b_n})f\big(\frac{k}{n} b_n\big),
\end{equation}
with $b_n$ a positive increasing sequence with the properties $\lim\limits_{n\rightarrow\infty} b_n =\infty$ and $\lim\limits_{n\rightarrow 
\infty}\frac{b_n}{n} = 0.$ They also studied some approximation properties of the operators.\\
Recently,  A. Karaisa \cite{Karaisa} defined Choldowsky type generalization of the Favard-Sz\'{a}sz operators as follows:
\begin{equation}\label{nit2}
P_{n}^{*}(f;q;x)=\frac{E_{q}^{\frac{[n]_q x}{b_n}}}{A(1)} \sum_{k=0}^{\infty} \frac{P_k(q; \frac{[n]_q x}{b_n})}{ [k]_q!} f\big(\frac{[k]_q}{[n]_q} b_n\big),
\end{equation}
where $\{P_k(q;.)\}\geq 0$ is a $q$-Appell polynomial set which is generated by

\begin{equation*}
A(u ) \frac{e^{\frac{[n]_q x}{b_n}}}u =  \sum_{k=0}^{\infty} \frac{P_k(q; \frac{[n]_q x}{b_n}) u^k}{ [k]_q!}
\end{equation*}
and $A(u)$ is defined by $A(u)=\sum_{k=0}^{\infty} a_k u^k.$\\
Motivated by these results, in this paper we study weighted approximation and error estimation of these operators.\\ 
During the last two decades, the applications of $q$-calculus emerged as a new area in the field of approximation theory. The rapid development of $q$-calculus has led to the discovery of various generalizations of Bernstein polynomials involving $q$-integers. The aim of these generalizations is to provide appropriate and powerful tools to application areas such as numerical analysis, computer-aided geometric design and solutions of differential equations.\\
To make the article self-content, here we mention certain basic definitions of $q$-calculus,
details can be found in \cite{KC} and the other recent articles. For each non negative
integer $n$, the $q$-integer $[n]_q$ and the $q$-factorial $[n]_q!$ are, respectively, defined by
\begin{equation*} \displaystyle [n]_q = \left\{ \begin{array}{ll} \frac{1-q^n}{1-q}, & \hbox{$q\neq1$}, \\
n,& \hbox{$q=1$}, \end{array} \right.
 \end{equation*}
and

\begin{eqnarray*}
[n]_q!=\left\{
\begin{array}{ll}
[n]_q[n-1]_q[n-2]_q...[1]_q,  & \hbox{$n=1,2,...$},\\
1,& \hbox{$n=0$.}
\end{array}\right. 
\end{eqnarray*}
Then for $q >0$ and integers $n, k, k \geq n \geq 0$, we have\\
$$[n+1]_{q}=1+q[n]_q  ~~~~{ \text {  and }}~~~ [n]_q+q^n[k-n]_q=[k]_q.$$
The $q$-derivative $D_qf$ of a function $f$ is defined by
$$(D_qf)(x)=\frac{f(x)-f(qx)}{(1-q)x}, x\neq 0.$$
The $q$-analogues of the exponential function are given by
$$e^{x}_q=\sum_{n=0}^{\infty} \frac{x^n}{[n]_q!},$$
and
$$E^{x}_q=\sum_{n=0}^{\infty} q^\frac{n(n-1)}{2} \frac{x^n}{[n]_q!}.$$ 
The exponential functions have the following properties:\\
$$D_q(e_{q}^{ax})=ae_{q}^{ax},~ D_q(E_{q}^{ax})=a E_{q}^{aqx},~ e_{q}^{x} E_{q}^{-x}= E_{q}^{x} e_{q}^{-x} =1.$$

\begin{lemma}\label{L1}
\cite{Karaisa} The following hold:\\
\begin{itemize}
\item[(i)] $P^{*}_{n}(e_0;q;x)=1,$
\item[(ii)] $P^{*}_{n}(e_1;q;x)=x + \frac{D_q(A(1)) E_{q}^{-\frac{[n]_q x}{b_n}} e_{q}^{ q \frac{[n]_q x}{b_n}}}{A(1)} \frac{b_n}{[n]_q},$
\item[(iii)] $P^{*}_{n}(e_2;q;x)= x^2 + \frac{ E_{q}^{-\frac{[n]_q x}{b_n}} e_{q}^{ q \frac{[n]_q x}{b_n}} [q D_q(A(q))+D_q(A(1)]}{A(1)} \frac{b_n x}{[n]_q} + \frac{D_q^2(A(1)) E_{q}^{-\frac{[n]_q x}{b_n}} e_{q}^{ q \frac{[n]_q x}{b_n}}}{A(1)} \frac{b_n^2}{[n]_q^2}.$
\end{itemize}
where $e_i(x)=x^i, i=0,1,2.$
\end{lemma}
Now we give an auxiliary lemma for the Korovkin test functions.

\begin{lemma}\label{L2}
\begin{itemize}
\item[(i)] $P^{*}_{n}(t-x ;q;x)= \frac{D_q(A(1)) E_{q}^{-\frac{[n]_q x}{b_n}} e_{q}^{ q \frac{[n]_q x}{b_n}}}{A(1)} \frac{b_n}{[n]_q},$
\item[(iii)] $P^{*}_{n}((t-x)^2 ;q;x)= \frac{ E_{q}^{-\frac{[n]_q x}{b_n}} e_{q}^{ q \frac{[n]_q x}{b_n}} [q D_q(A(q))-D_q(A(1)]}{A(1)} \frac{b_n x}{[n]_q} + \frac{D_q^2(A(1)) E_{q}^{-\frac{[n]_q x}{b_n}} e_{q}^{ q \frac{[n]_q x}{b_n}}}{A(1)} \frac{b_n^2}{[n]_q^2} .$
\end{itemize}
\end{lemma}

\section{Weighted approximation}
Let $B_{x^2}[0,\infty)$ be the set of functions defined on $[0,\infty)$ satisfying the condition $\mid f(x) \mid \leq
M_f(1+x^2)$, where $M_f$ is a constant depending on  $f$ only. By $C_{x^2}[0,\infty)$, we denote subspace of all continuous functions belonging to $B_{x^2}[0,\infty)$. Also, let $C_{x^2}^*[0,\infty)$ be the subspace of all $f\in
C_{x^2}[0,\infty)$ for which $\lim\limits_{x \rightarrow \infty}\frac{f(x)}{1+x^2}$ is finite. The norm on
$C_{x^2}^*[0,\infty)$ if $\|f\|_{x^2}=\sup\limits_{x\in[0,\infty)} \frac{\mid f(x)\mid}{1+x^2}$. For any positive number $a$, we define $$\omega_a(f,\delta)=\sup_{\mid t-x\mid \leq \delta}\sup_{x,t \in [0,a]} \mid f(t)-f(x)\mid,$$ and denote the usual modulus of continuity of $f$ on the closed interval $[0, a]$. We know that for a function $f \in C_{x^2}[0,\infty)$, the modulus of continuity $\omega_a(f,\delta)$ tends to zero.\\
Now, we shall discuss the weighted approximation theorem, when the approximation formula holds true on the interval $[0,\infty)$.

\begin{theorem}
For each $f \in C_{x^2}^{*}[0,\infty)$, we have
\begin{equation*}
\lim _{n\rightarrow \infty} \| P_{n}^{*}(f;q;x) - f \|_{x^2} = 0.\end{equation*}
\end{theorem}

\begin{proof}
Using the theorem in \cite{gz} we see that it is sufficient to verify the following three conditions
\begin{equation}\label{l6}
\lim_{n\rightarrow \infty} \| P_{n}^{*}(t^r;q;x)-x^r\|_{x^2}=0, \text{  } r=0,1,2.
\end{equation}
Since, $P_{n}^{*}(1,x)=1$, the first condition of (\ref{l6}) is satisfied for $r=0$. Now,
\begin{eqnarray*}
\|P_{n}^{*}(t;q;x)-x\|_{x^2}&=&\sup_{x\in [0,\infty)} \frac{\mid P_{n}^{*}(t;q;x)-x \mid}{1+x^2}\\
&\leq &  \sup_{x\in[0,\infty)}\bigg|\left(x + \frac{D_q(A(1)) E_{q}^{-\frac{[n]_q x}{b_n}} e_{q}^{ q \frac{[n]_q x}{b_n}}}{A(1)} \frac{b_n}{[n]_q}- x\right) \frac{1}{1+x^2}\bigg|\\
&\leq & \sup_{x\in[0,\infty)}\bigg|\left( \frac{D_q(A(1)) E_{q}^{-\frac{[n]_q x}{b_n}} e_{q}^{ q \frac{[n]_q x}{b_n}}}{A(1)} \frac{b_n}{[n]_q} \right) \frac{1}{1+x^2}\bigg|
\end{eqnarray*}
which implies that $$\|P_{n}^{*}(t,x)-x\|_{x^2}=0.$$
Finally,
\begin{eqnarray*}
\|P_{n}^{*}(t^2;q;x)-x^2\|_{x^2}&=& \sup_{x\in [0,\infty)} \frac{\mid P_{n}^{*}(t^2;q;x)-x^2 \mid}{1+x^2}\\
&\leq & \sup_{x\in [0,\infty)}\left| x^2 + \frac{ E_{q}^{-\frac{[n]_q x}{b_n}} e_{q}^{ q \frac{[n]_q x}{b_n}} [q D_q(A(q))+D_q(A(1)]}{A(1)} \frac{b_n x}{[n]_q} + \frac{D_q^2(A(1)) E_{q}^{-\frac{[n]_q x}{b_n}} e_{q}^{ q \frac{[n]_q x}{b_n}}}{A(1)} \frac{b_n^2}{[n]_q^2} - x^2\right|\frac{1}{1+x^2}\\
\end{eqnarray*}
which implies that
$\|P_{n}^{*}(t^2;q;x)-x^2\|_{x^2} \rightarrow 0$ as $[n]_q \rightarrow \infty.$
Thus proof is completed.
\end{proof}
We give the following theorem to approximate all functions in $C_{x^2}[0,\infty)$.

\begin{theorem}
For each $f\in C_{x^2}[0,\infty)$ and $\alpha>0$, we have
$\lim\limits_{[n]_q \rightarrow \infty} \sup\limits_{x\in[0,\infty)} \frac{\mid P_{n}^{*}(f;q;x)-f(x) \mid}{(1+x^2)^{1+\alpha}}=0.$
\end{theorem}

\begin{proof}
For any fixed $x_0>0$,
\begin{eqnarray*}
\sup_{x\in[0,\infty)}\frac{\mid P_{n}^{*}(f;q;x)-f(x) \mid}{(1+x^2)^{1+\alpha}}&\leq& \sup_{x \leq x_0}\frac{\mid P_{n}^{*}(f,x)-f(x) \mid}{(1+x^2)^{1+\alpha}} + \sup_{x \geq x_0}\frac{\mid P_{n}^{*}(f;q;x)-f(x) \mid}{(1+x^2)^{1+\alpha}}\\
&\leq&  \|P_{n}^{*}(f;q;x)-f\|_{C[0,x_0]} + \|f\|_{x^2}\sup_{x \geq x_0}\frac{\mid P_{n}^{*}(1+t^2;q;x)\mid}{(1+x^2)^{1+\alpha}}+\sup_{x \geq x_0}\frac{\mid f(x) \mid}{(1+x^2)^{1+\alpha}}.
\end{eqnarray*}
The first term of the above inequality tends to zero from Theorem \ref{t2}. By Lemma \ref{L2} for any fixed $x_0>0$ it is easily seen that $ \sup_{x\geq x_0} \frac{\mid P_{n}^{*}(1+t^2;q;x)\mid}{(1+x^2)^{1+\alpha}}$ tends to zero as $ [n]_q \rightarrow \infty$. We can choose $x_0>0$ so large that the last part of the above inequality can be made small enough. Thus the proof is completed.
\end{proof}

\section{Error Estimation}
The usual modulus of continuity of $f$ on the closed interval $[0, b]$ is defined
by
$$\omega_b(f,\delta) =\sup_{|t-x|\leq\delta,\, x,t\in[0,b]}|f(t)-f(x)|,\,\,  b>0.$$
We first consider the Banach lattice, for a function $f\in E$, $ \lim_{\delta\rightarrow 0^+}\omega_b(f,q;\delta)=0,$
where
$$E:=\left\{f\in C[0,\infty):\lim_{x\rightarrow\infty}\frac{f(x)}{1+x^2}\,\, is\,\, finite \right\}.$$
The next theorem gives the rate of convergence of the operators $P_{n}^{*}(f,x)$ to  $f(x),$ for all $f \in E.$ 

\begin{theorem}  \label{t2}
Let $f\in E$ and $\omega_{b+1}(f,q;\delta)$, $0<q<1$ be its modulus of continuity on the
finite interval $[0,b+1]\subset[0,\infty)$, where $a>0$ then, we have
\begin{equation*}
\| P_{n}^{*}(f;q;x)-f \|_{C[0,b]} \leq
M_f(1+b^2)\delta_n(b)+2\omega_{b+1}\left(f,\sqrt{\delta_n(b)}\right).
\end{equation*}
\end{theorem}

\begin{proof}
The proof is based on the following inequality
\begin{equation}\label{t3}
\|P_{n}^{*}(f;q;x)-f \| \leq
M_f(1+b^2)P_{n}^{*}((t-x)^2,x)+
\left(1+\frac{P_{n}^{*}(|t-x|,x)}{\delta}\right)\omega_{b+1}(f,\delta).
\end{equation}
For all $(x,t)\in [0,b]\times[0,\infty):= S.$
To prove (\ref{t3}), we write\\
$$S=S_1\cup S_2:=\{(x,t):0\leq x\leq b,\, 0\leq t \leq b+1\}\cup\{(x,t):0\leq x\leq b,\,  t> b+1\}.$$
If $(x, t)\in S_1,$ we can write
\begin{equation}\label{o}
|f(t)-f(x)|\leq \omega_{b+1}(f,|t-x|)\leq\left(1+\frac{|t-x|}{\delta}\right)\omega_{b+1}(f,\delta)
\end{equation}
where $\delta > 0.$ On the other hand, if $(x, t)\in S_2,$ using the fact that $t-x > 1$,
we have
\begin{eqnarray}\label{o1}
|f(t)-f(x)|&\leq& M_f(1+x^2+t^2)\\
&\leq& M_f(1+3x^2+2(t-x)^2)\nonumber\\
&\leq& N_f(1+b^2)(t-x)^2\nonumber
\end{eqnarray}
where $N_f = 6M_f.$ Combining (\ref{o}) and (\ref{o1}), we get (\ref{t3}).
Now from (\ref{t3}) it follows that
\begin{eqnarray*}
|P_{n}^{*}(f;q;x)-f(x)|&\leq & N_f(1+b^2)P_{n}^{*}((t-x)^2;q;x) + \left(1+\frac{P_{n}^{*}(|t-x|,x)}{\delta}\right)\omega_{b+1}(f,\delta)\\
&\leq& N_f(1+b^2)P_{n}^{*}((t-x)^2,x) + \left(1+\frac{{[P_{n}^{*}((t-x)^2,x)]}^{1/2}}{\delta}\right)\omega_{b+1}(f,\delta).
\end{eqnarray*}
By Lemma \ref{L2}  we have
$$P_{n}^{*}(t-x)^2\leq\delta_n(b).$$
\begin{equation*}
\|P_{n}^{*}(f;q;x)-f \| \leq
N_f(1+b^2)\delta_n(b)+\left(1+\frac{\sqrt{\delta_n(b)}}{\delta}\right)\omega_{b+1}(f,\delta).
\end{equation*}
Choosing $\delta =\sqrt{\delta_n(b)},$ we get the desired estimation.
\end{proof}

\end{document}